\documentclass[12pt]{amsart}
\usepackage{amsfonts, amssymb, amsmath, amsthm, euscript}
\swapnumbers
\theoremstyle{plain}
\newtheorem{thm}{Theorem}[section]
\newtheorem{lem}[thm]{Lemma}
\newtheorem{prop}[thm]{Proposition}
\newtheorem{cor}[thm]{Corollary}

\theoremstyle{definition}
\newtheorem{setup}[thm]{Setup}
\newtheorem{exam}[thm]{Example}
\newtheorem{exams}[thm]{Examples}

\theoremstyle{remark}

\newtheorem{note}[thm]{}

\def\qij{q_{ij}}
\def\hij{h_{ij}}
\def\nxn{n{\times}n}
\def\dxd{d{\times}d}
\def\pxp{p{\times}p}
\def\N{\mathbb{N}}
\def\Nn{\mathbb{N}^n}
\def\Z{\mathbb{Z}}
\def\Zn{\mathbb{Z}^n}

\def\kq{k_q}

\def\ktimes{k^{\times}}

\def\s{{\bf s}}
\def\b{{\bf b}}

\begin{document}

\title[On $q$-commutative power and Laurent series]{On
  $q$-commutative power and Laurent \\ series rings at roots of unity}

\author{Edward S. Letzter}

\address{Department of Mathematics\\
        Temple University\\
        Philadelphia, PA 19122}
      
      \email{letzter@temple.edu }

\author{Linhong Wang}

\address{Department of Mathematics\\
University of Pittsburgh\\
Pittsburgh, PA 15260}

\email{lhwang@pitt.edu}

\author{Xingting Wang}

\address{Department of Mathematics\\
        Temple University\\
        Philadelphia, PA 19122}

\email{xingting@temple.edu}

\begin{abstract} We continue the first and second authors' study of
  $q$-commutative power series rings $R$ $=$ $\kq[[x_1,\ldots,x_n]]$
  and Laurent series rings $L$\-$=$
  $\kq[[x^{\pm 1}_1,\ldots,x^{\pm 1}_n]]$, specializing to the case in
  which the commutation parameters $\qij$ are all roots of unity. In
  this setting, $R$ is a PI algebra, and we can apply results of De
  Concini, Kac, and Procesi to show that $L$ is an Azumaya algebra
  whose degree can be inferred from the $\qij$. Our main result
  establishes an exact criterion (dependent on the $\qij$) for
  determining when the centers of $L$ and $R$ are commutative Laurent
  series and commutative power series rings, respectively. In the
  event this criterion is satisfied, it follows that $L$ is a unique
  factorization ring in the sense of Chatters and Jordan, and it
  further follows, by results of Dumas, Launois, Lenagan, and Rigal,
  that $R$ is a unique factorization ring. We thus produce new
  examples of complete, local, noetherian, noncommutative, unique
  factorization rings (that are PI domains).
\end{abstract}



\maketitle 


\section{Introduction}

Given a multiplicatively antisymmetric $\nxn$ matrix $q = (\qij)$,
over an algebraically closed field $k$, one can construct the
$q$-commutative power and Laurent series rings
$R = \kq[[x_1,\ldots,x_n]]$ and
$L = \kq[[x^{\pm 1}_1,\ldots,x^{\pm 1}_n]]$, with multiplication 
determined by the relations $x_ix_j = \qij x_j x_i$.  The ring $R$ is
a complete, local, noetherian, Auslander regular, zariskian (in the
sense of \cite{LiVOy}) domain, whose global dimension, Krull
dimension, and classical Krull dimension are all equal to $n$ (see
e.g.~\cite{LetWan}). In analogy with the commutative case, the ring
$L$ is the localization of $R$ at the normal elements
$x_1,\ldots,x_n$. While there is now a deep and detailed structure
theory available for quantizations of polynomial rings and related
algebras (see, e.g., \cite{GooYak} for one recent study), there is
still much to learn about $q$-commutative power series rings and
related algebras. (The reader is referred e.g.~to \cite{Koo} for a
survey on the combinatorics of $q$-commutative power series.)

Our focus in this note is on the case in which the $\qij$ are all
roots of unity. (Recall that the theories of quantum groups and
related algebras generally divide into the``root of unity'' and ``not
a root of unity'' cases.)  Our analysis builds on the previous study
by the first and second author \cite{LetWan} and depends on earlier
results of Kac, De Concini, and Procesi concerning ``twisted
polynomial algebras'' at roots of unity \cite{DeCKacPro}.

Assuming that the $\qij$ are all roots of unity, it follows from
standard reasoning that $R$ is a polynomial identity (PI) algebra. In
(\ref{Azumaya}), we use \cite{DeCKacPro} to show that $L$ is an
Azumaya algebra and that the PI degree of $R$ and $L$ can be inferred
from the $\qij$.

Continuing to assume that the $\qij$ are all roots of unity, our main
result (\ref{center-theorem}) provides an exact criterion (dependant
on the $\qij$) for determining when the center of $L$ is a Laurent
series ring in central indeterminates $z_1,\ldots,z_n$ and when the
center of $R$ is a power series ring in $z_1,\ldots,z_n$. Morever,
when this criterion is satisfied, it follows that $L$ is a unique
factorization ring in the sense of Chatters and Jordan
\cite{ChaJor}. Furthermore, when $L$ is a unique factorization ring,
it follows from the noncommutative analog of Nagata's Lemma given by
Dumas and Rigal \cite{DumRig}, and Launois, Lenagan, and Rigal
\cite{LauLenRig}, that $R$ must also be a unique factorization
ring. We thus obtain new examples of PI domains that are unique
factorization rings. We do not know of a precise description of the
centers of $R$ and $L$ when the criterion does not hold.

\section{Roots of unity, polynomial identities, and Azumaya algebras}

The reader is referred (e.g.) to \cite{GooWar}, \cite{LiVOy}, and
\cite{McCRob} for further background on noetherian rings, PI rings,
and filtered rings.

\begin{setup} \label{Setup} (Following \cite[\S 2]{LetWan}, which in
  turn followed \cite[\S 4]{BroGoo1}, \cite{DeCKacPro}, \cite[\S
  1]{GooLet}, and \cite{Hod}.)

  (i) Assume throughout that $k$ is an algebraically closed field,
  that $\ktimes$ denotes the multiplicative group of units in $k$, and
  that $n$ denotes a positive integer. Let $q = (\qij)$ be an
  $\nxn$ matrix with entries in $\ktimes$ such that $q_{ii} = 1$ and
  $\qij = q^{-1}_{ji}$ for all $1 \leq i,j \leq n$. The set of
  non-negative integers will be denoted $\N$.

  (ii) The $q$-commutative power series ring
  $R = \kq[[x_1,\ldots,x_n]]$ is the associative unital $k$-algebra of
  formal (skew) power series, in indeterminates $x_1,\ldots,x_n$,
  subject only to the commutation relations $x_i x_j = \qij x_j
  x_i$. The elements of $R$, then, are formal power series
 \[ \sum _{s \in \Nn} c_s x^s ,\]
 for $c_s \in k$, for $s = (s_1,\ldots,s_n) \in \Nn$, and for $x^s =
 x_1^{s_1}\cdots x_n^{s_n}$. 

 (iii) Contained within $R$ is the $k$-subalgebra
 $R' = \kq[x_1,\ldots,x_n]$ of polynomials in the $x_1,\ldots,x_n$,
 subject to the same commutation relations as in $R$.
\end{setup}

\begin{note} \label{J-adic} (Following \cite[\S 2]{LetWan}; the reader
  is referred to \cite{LiVOy} for further information on filtrations
  and completions.)

  (i) $R$ is the completion of $R'$ at the ideal generated by
  $x_1,\ldots,x_n$.  

  (ii) Letting $J$ denote the augmentation ideal
  $\langle x_1,\ldots , x_n \rangle$ of $R$, it follows that the
  $J$-adic filtration
\[J^0 \supset J^1 \supset J^2 \cdots \]
is separated (i.e., the intersection of the $J^i$ is the zero ideal)
and complete (i.e., Cauchy sequences converge in the $J$-adic
topology). 

(iii) $R$ is a complete local ring with unique primitive ideal $J$.

(iv) The associated graded ring of $R$ corresponding to the $J$-adic
filtration is isomorphic to the noetherian domain $R'$, and it
follows that $R$ is a noetherian domain.

(v) $R'$ is dense in $R$ under the $J$-adic topology.
\end{note}

\begin{note} It is a standard fact that $R'$ satisfies a polynomial
  identity if and only the $\qij$ are all roots of
  unity. Consequently, if the $\qij$ are \emph{not} all roots of
  unity, then the over-ring $R$ of $R'$ cannot satisfy a PI. On the
  other hand, if the $\qij$ \emph{are} all roots of unity, then there
  is some positive integer $\gamma$ such that
  $x_1^\gamma, \ldots, x_n^\gamma$ are all central in $R$, making $R$
  a finite module over the central subalgebra
  $k[[x^\gamma_1,\ldots, x_n^\gamma]]$.  We conclude that $R$
  satisfies a polynomial identity (and is a finite module over its
  center) if and only if the $\qij$ are all roots of unity.
\end{note}

We turn next to $q$-commutative Laurent series rings. 

\begin{note} Let $X$ denote the multiplicatively closed subset of $R$
  generated by $1$ and the indeterminates $x_1,\ldots,x_n$.  Since 
  each $x_i$ is normal in $R$, we see that $X$ is an Ore set in 
  $R$. Henceforth, we will let $L$ denote the \emph{$q$-commutative 
    Laurent series ring} $\kq[[x^{\pm 1}_1,\ldots,x^{\pm 1}_n]]$,
  obtained via the Ore localization of $R$ at $X$. (See, e.g.,
  \cite{GooWar} or \cite{McCRob} for details on Ore localizations.) 
  Note that $L$ is a noetherian domain and can be obtained as the localization of $R$ at 
  the product $x_1\cdots x_n$, which is a single normal element of $R$. 

  Each element of $L$ will have the form 
 \[ \sum _{s \in \Zn} c_s x^s ,\]
 for $c_s \in k$, for $s = (s_1,\ldots,s_n) \in \Zn$, and for $x^s =
 x_1^{s_1}\cdots x_n^{s_n}$, but with 
\[c_s = 0 \quad \text{for} \quad \min\{s_1,\ldots,s_n\} \ll 0.\]

We will let $L' = \kq[x^{\pm 1}_1,\ldots,x^{\pm 1}_n]$ denote the well-known
subalgebra of $L$ comprised of $q$-commutative Laurent polynomials in
the $x_1,\ldots, x_n$.
\end{note}

\begin{note} \label{DKP} Assuming the $\qij$ to all be roots of unity,
  the representation theory of $L'$ is described by De Concini, Kac,
  and Procesi as follows \cite{DeCKacPro}: Let $\ell$ be the order of,
  and let $\epsilon$ be a generator of, the multiplicative cyclic
  group $\langle \qij \rangle$. Choose $\hij$, for
  $1 \leq i,j \leq n$, such that $\qij = \epsilon ^{\hij}$. Then the
  matrix $(\hij)$ defines a homomorphism
\[H:\Zn \longrightarrow (\Z/\ell \Z)^n .\]
Letting $h$ denote the cardinality of the image of $H$, it is
proved in \cite{DeCKacPro} that $R'$ has PI degree $\sqrt{h}$ and that
$L'$ is an Azumaya algebra of PI degree $\sqrt{h}$.
\end{note}

\begin{lem} \label{converges} Let $p$ be a positive integer, and
  let $\s= \s(X_1,X_2,\ldots)$ be a multilinear polynomial identity for
  $\pxp$-matrices. Then $\s$ vanishes on $R$ if and only if $\s$
  vanishes on $R'$. 
\end{lem}

\begin{proof} Assume that $\s$ vanishes on $R'$; to verify the lemma
  it suffices to show that $\s$ also vanishes on $R$. So choose
  $f_1,f_2,\ldots \in R$, and let $N$ be a positive integer. Since
  $R'$ is dense in $R$ under the $J$-adic topology, for each $i$ we
  can write
\[ f_i = a_{i,N} + b_{i,N} ,\]
for some $a_{i,N} \in R'$ and $b_{i,N} \in J^N$. 

Next, noting that $\s$ is multilinear on $R$, that $\s$ is a polynomial
identity for $R'$, and that $J$ is an ideal of $R$, we have:
\[\s(f_1,f_2,\ldots) = s(a_{1,N} + b_{1,N}, a_{2,N} + b_{2,N}, \ldots) \in
\s(a_{1,N}, a_{2,N}, \ldots) + J^N = 0 + J^N = J^N.\]
Therefore,
\[0 = \lim_{N \rightarrow \infty} \s(a_{1,N},a_{2,N},\ldots) =
\s(f_1,f_2,\ldots) .\]
Hence $\s$ vanishes on $R$, and the lemma follows.
\end{proof}

\begin{note} \label{Assumptions} \emph{For the remainder of this note,
    assume that the $\qij$ are all roots of unity and that $h$ is as
    defined in {\rm (\ref{DKP})}; set $d = \sqrt{h}$.}
\end{note}

\begin{prop} \label{Azumaya} {\rm (i)} $R$ has PI degree $d$. {\rm
    (ii)} $L$ is an Azumaya algebra of PI degree 
  $d$.
\end{prop}

\begin{proof} (i) This follows by combining (\ref{DKP}) and
  (\ref{converges}). 

(ii)  We know from (i) that the PI degree of $R$ is $d$, and since $R$ and
  $L$ are both right orders in the same Goldie quotient domain, it
  follows that $L$ also has PI degree $d$.

  Following (e.g.) \cite[\S 13.5]{McCRob}, let $g_d$ denote the
  multilinear central polynomial for $\dxd$ matrices constructed in
  \cite[13.5.11]{McCRob}. Then if $A$ is an arbitrary prime ring, the
   Artin-Procesi theorem ensures that $A$ is an Azumaya algebra of
  degree $d$ if and only if $g_d(A)A = A$; see \cite[13.7.14]{McCRob}.

We saw in (\ref{DKP}) that $L'$ is an Azumaya algebra of
degree $d$, and so $g_d(L')L' = L'$. In particular, $1 \in g_d(L')L'
\subseteq g_d(L)L$. Hence $g_d(L)L = L$ and $L$ is an Azumaya algebra
of degree $d$.
\end{proof}

\section{Centers}

Retain the assumptions and notation of the preceding section,
particularly the assumption that all of the $\qij$ are roots of unity.

\begin{note} \label{center} (See, e.g., \cite[\S 1]{GooLet}.)

  (i) Consider the alternating bicharacter
  $\sigma\colon\Zn{\times}\Zn \rightarrow \ktimes$ defined by
\[ \sigma(s,t)  =  \prod_{i,j = 1}^n \qij^{s_it_j} ,\]
for $s = (s_1,\ldots,s_n)$ and $t = (t_1,\ldots,t_n)$ in $\Zn$.  Then
\[ x^s x^t = \sigma(s,t)x^t x^s .\]

(ii) Set
\[ S  =  \{ s \in \Zn  \mid  \text{$\sigma(s,t) = 1$ for
  all $t \in \Zn$} \}.\]

Note that $S$ is a (free abelian) subgroup of $\Z^n$. 

(iii) A monomial $x^s$, for $s \in \Zn$, commutes with all monomials
$x^t$, for all $t \in \Zn$, if and only if $s \in S$. Letting $Z$
denote the center of $L$, it follows that
\[ Z = \left.\left\{ \; \sum_{s \in S, \; \nu_s 
    \in k} \nu_sx^s \in L \; \right| \; \nu_s = 0 \;  \text{for} \; \min\{s_1,\ldots,s_n\}
\ll  0 \; \right\}.\]
\end{note}

\begin{note} \label{positive} (i) Say that an $n$-tuple $s \in \Zn$ is \emph{positive
    (with respect to the standard basis)} if all of the entries of $s$
  are non-negative and $s$ contains at least one positive entry.

(ii) Note that $s \in \Zn$ is positive if and only if 
\[ \sum_{i=0}^\infty c_i (x^s)^i ,\]
is a well defined element of $L$, for all choices of $c_i \in k$,
where $(x^s)^i$ denotes the product of $x^s$ with itself $i$ times.

(iii) Let $\ell$ be the order of the finite multiplicative 
  group generated by the $\qij$. Then the $n$-tuples 
  \[ (\ell,0,\ldots,0), (0,\ell,0,\dots,0), \ldots , (0, \dots , 0,
  \ell) \]
  form a $\Z$-linearly independent subset of $S$, and so the abelian
  group $S$ has rank $n$.
\end{note}

\begin{note} \label{positive-diagonal} Let $\b_1,\ldots,\b_n$ be a $\Z$-basis for $S$, and let
  $B$ denote the $\nxn$ matrix whose $i$th row is $\b_i$.

  (i) We will say that the basis $\b_1,\ldots,\b_n$ is \emph{positive
    diagonal} if $B$ is a diagonal matrix whose diagonal entries are
  all positive. 

  (ii) If a given  $s \in S$ can be written as
  \[s = m_1 \b_1 + \cdots + m_n \b_n,\]
  for non-negative $m_1,\ldots,m_n \in \Z$, and with at least one
  $m_i > 0$, then we will say that $s$ is \emph{positive with respect to
    $B$}.
\end{note}

\begin{thm} \label{center-theorem}  Let $\b_1,\ldots,\b_n$ be a 
  $\Z$-basis for $S$, and set 
\[z_1 = x^{\b_1}, \ldots , z_n = x^{\b_n} .\]
The following are equivalent: {\rm (i)} $\b_1,\ldots,\b_n$ is a
positive diagonal basis for $S$ (after reordering if necessary). {\rm (ii)}
$Z = Z(L)$ is a Laurent series ring over $k$ in $z_1,\ldots, z_n$. {\rm (iii)}
$Z(R)$ is a power series ring over $k$ in $z_1,\ldots , z_n$.
\end{thm}

\begin{proof} First note, for $m_1,\ldots,m_n \in \Z$ and 
\[ s = m_1\b_1 + \cdots + m_n\b_n  = (s_1,\ldots,s_n) \in S,\]
that 
\[ x^s = x^{m_1\b_1 + \cdots + m_n\b_n} = \gamma_s z_1^{m_1}\cdots 
z_n^{m_n}, \eqno{(\ast)}\]
for some nonzero $\gamma_s \in k$. 

We now prove (i) $\Rightarrow$ (ii) and (i) $\Rightarrow$ (iii). To
start, assume that
\[\b_1 = (\lambda_1,0,\ldots,0), \ldots, \b_n = (0,\dots , 0, \lambda_n), \]
for positive $\lambda_1,\ldots,\lambda_n \in \Z$. Consequently,
$s= (m_1\lambda_1,\ldots,m_n\lambda_n)$ will be positive with respect
to the standard basis if and only if $s$ is positive with respect to
$B$. Moreover,
\[ \min\{s_1,\ldots, s_n\} \; \ll \; 0 \; \Longleftrightarrow \;
\min\{m_1,\ldots , m_n\} \; \ll \; 0 .\]
Setting $m = (m_1,\ldots,m_n)$, it follows from (\ref{center}.iii) and
($\ast$) that
\[ \begin{aligned} Z &= \left.\left\{ \; \sum_{s \in S, \; \nu_s \in 
        k} \nu_sx^s \; \right| \; \nu_s = 0 \; \text{for} \; \min\{s_1,\ldots,s_n\}
    \ll 0 \; \right\} \\
  &= \left.\left\{ \; \sum_{m \in \Zn, \;  \mu_{m} \in k}
      \mu_{m} z_1^{m_1}\cdots z_n^{m_n} \; \right| \; \mu_{m} = 0 \;
    \text{for} \; \min\{m_1,\ldots,m_n\} \ll  0 \; \right\} \\
  &= k[[z_1^{\pm 1}, \ldots , z_n^{\pm 1}]]. \end{aligned}\]

Next observe that the center of $R$ must equal $Z \cap R$, and that
\[ \begin{aligned} Z \cap R &= \left.\left\{ \; \sum_{s \in S, \; \nu_s \in 
        k} \nu_sx^s \; \right| \; \nu_s = 0 \; \text{if $s_i < 0$ for
      some $i = 1,\ldots, n$} \; \right\} \\
  &= \left.\left\{ \; \sum_{m \in \Zn, \;  \mu_{m} \in k}
      \mu_{m} z_1^{m_1}\cdots z_n^{m_n} \; \right| \; \mu_{m} = 0 \;
\text{if $m_i < 0$ for
      some $i = 1,\ldots, n$} \; \right\} \\
  &= k[[z_1, \ldots , z_n]]. \end{aligned}\]
Therefore, (i) $\Rightarrow$ (ii), and (i) $\Rightarrow$ (iii).

To prove (ii) $\Rightarrow$ (i) and (iii) $\Rightarrow$ (i), first
assume that arbitrary Laurent series over $k$ in $z_1,\ldots,z_n$ form
a well defined subalgebra of $Z$.  Choose
$s = m_1\b_1 + \cdots + m_n \b_n \in S$. Then by (\ref{positive}) and
($\ast$),
\[ \sum_{i=0}^\infty c_i (x^s)^i = \sum_{i=0}^\infty c_i \left(\gamma_s z_1^{m_1}\cdots 
z_n^{m_n}\right)^i\]
will be a well defined element of $L$, for all choices of $c_i\in k$,
if and only if $s$ is positive with respect to $B$, if and only if $s$
is positive with respect to the standard basis. In particular, each
$\b_i$ is positive with respect to the standard basis, and so the
entries of $B$ are all non-negative.

Since $B$ has rank $n$, we can use $\Z$-linear Gaussian elimination to
produce an $\nxn$ matrix $A$, with integer entries, such that 
\[D := AB = \begin{bmatrix} \lambda_1 & & \\
  &\ddots& \\
  & & \lambda_n
\end{bmatrix}\]
for positive integers $\lambda_1,\ldots,\lambda_n$. 

Since each row of $D$ is positive with respect to the standard basis,
it follows from above that each row of $D$ is positive with respect to
$B$.  Therefore, all of the entries of $A$ must be
non-negative. However, if the entries of $A$ and $B$ are all
non-negative, and if $AB$ is a diagonal matrix, then it must be the
case that $B$ has exactly one nonzero entry in each row. We now see
that $B$ is positive diagonal (after changing the order of the rows,
if necessary).

The proof of (ii) $\Rightarrow$ (i) follows. 

To prove (iii) $\Rightarrow$ (i), assume that $Z(R) =
k[[z_1,\ldots,z_n]]$. Note that the Laurent series ring
$k[[z_1^{\pm 1},\ldots, z_n^{\pm 1}]]$ will then be a well defined
$k$-subalgebra of $Z \subseteq L$, and so the above argument can be
applied.

The theorem follows.
\end{proof}

\begin{exams} \label{example} Suppose that the characteristic of $k$ is not equal to
  $2$.

(i) Consider $R$ and $L$, for $n = 3$, with
  $x_ix_j = - x_jx_i$ for $i \ne j$. A monomial
  $x_1^{s_1}x_2^{s_2}x_3^{s_3}$ will be in the center of $L$ if and
  only if either $|s_1| = |s_2| = |s_3|$ or $s_1$, $s_2$, $s_3$ are
  all even. It follows that no basis for $S$ can be positive diagonal,
  and so the centers of $R$ and $L$ cannot be power or Laurent series
  rings, respectively. 

  (ii) Again consider $R$ and $L$ for $n = 3$, with
  $x_1x_2 = - x_2x_1$, $x_1x_3 = - x_3x_1$, and $x_2x_3 = x_3x_2$.
  Then a monomial $x^{s_1}x^{s_2}x^{s_3}$ is in the center of $L$ if
  and only if each of $s_1$, $s_2$, $s_3$ is even. Therefore,
  $(2,0,0)$, $(0,2,0)$, and $(0,0,2)$ form a basis for $S$. The center
  of $L$ is $k[[x_1^{\pm 2},x_2^{\pm 2},x_3^{\pm 2}]]$, and the center
  of $R$ is $k[[x_1^2,x_2^2,x_3^2]]$

  (iii) Let $n = 2$ and let $q$ be a primitive $\ell$th root of unity.
  The center of $L$ is $k[[x_1^{\pm \ell}, x_2^{\pm \ell}]]$, and the
  center of $R$ is $k[[x_1^\ell, x_2^\ell]]$. It was already proved in
  \cite{VDBVGa} that the center of an Ore-extension-based skew power
  series ring in two variables is a commutative power series ring in
  two variables.
\end{exams}
  
\section{Unique factorization}

Chatters and Jordan defined a \emph{unique factorization ring (UFR)}
to be a noetherian prime ring in which every height-one prime ideal is
principal \cite{ChaJor}. 

Recall our assumption that the $\qij$ are all roots of unity, and
recall our definition of \emph{positive diagonal} in
(\ref{positive-diagonal}).

\begin{prop} \label{L-ufr} If $S$ has a positive diagonal basis then $L$ is a UFR.
\end{prop}

\begin{proof} It follows from \cite[3.11]{LetWan} or from $L$ being
  an Azumaya algebra with center $Z$, as seen in (\ref{Azumaya}.ii),
  that
\[ I = L(I \cap Z) \quad \text{and} \quad K = (LK)\cap Z ,\]
for all ideals $I$ of $L$ and ideals $K$ of $Z$. 

So let $P$ be a height-one prime ideal of $L$, and set $Q = P \cap Z$.
It follows from the preceding paragraph, or from \cite[3.14]{LetWan},
that $Q$ is a height-one prime ideal of $Z$.

By (\ref{center-theorem}), $Z$ is a commutative Laurent series ring,
and commutative Laurent series rings are UFDs (in the classical sense
of the term). Therefore, $Q = aZ$ for some element $a \in Z$. But then
\[ P = P(P\cap Z) = LQ = L(Za) = La.\]
Hence $P$ is principal, and so $L$ is a UFR.
\end{proof}

Combining the preceding proposition with the noncommutative analog of
Nagata's Lemma given by Dumas and Rigal \cite{DumRig}, and Launois,
Lenagan, and Rigal \cite{LauLenRig}, we obtain:

\begin{cor} If $S$ has a positive diagonal basis then $R$ is a UFR.
\end{cor}

\begin{proof} This follows directly from (\ref{L-ufr}) and
  \cite[1.4]{LauLenRig}, noting that $L$ is obtained from $R$ via
  localization at the nonzero normal element $x_1\cdots x_n \in R$.
\end{proof} 

\begin{exam} In \cite{Cha}, Chatters uses the term \emph{unique
    factorization domain} to refer to unique factorization rings (in
  the above sense) for which the zero ideal and all height-one prime
  ideals are completely prime.

  To see that $q$-commutative power and Laurent series rings, at roots of
  unity, are generally not unique factorization domains even when they
  are unique factorization rings, consider the case when $n = 2$, when
  $k$ has characteristic $\ne 2$, and when $q = -1$. The center of $L$
  is then $k[[x_1^{\pm 2}, x_2^{\pm 2}]]$. It is not hard to
  check that $x_1^2 + x_2^2 \in Z$ generates a height-one prime ideal
  of $Z$, and so
\[P = L(x_1^2 + x_2^2)\]
is a height-one prime ideal of $L$, by (e.g.)
\cite[3.14]{LetWan}. But, $(x_1 + x_2) \notin P$, and 
\[(x_1 + x_2)^2 \; = \; x_1^2 + x_1x_2 + x_2x_1 + x_2^2 = x_1^2 +
x_2^2 \; \in \; P.\]
Therefore, $P$ is a height-one prime ideal of $L$ that is not
completely prime, and $L$ is not a unique factorization
domain. Furthermore, routine reasoning now shows that $x_1^2 + x_2^2$
generates a height-one prime ideal of $R$ that is not completely
prime, and so $R$ also is not a unique factorization domain.
\end{exam}


\end{document}